\documentclass[final,leqno]{siamltex}
\usepackage{epsfig}
\usepackage{amsmath}
\usepackage{amssymb}
\usepackage{tikz}
\usepackage{graphicx}
\usepackage{float}
\usepackage[notcite,notref]{showkeys}
\newtheorem{algorithm}{Weak Galerkin Algorithm}
\newcommand{\bq}{{\bf q}}
\newcommand{\bn}{{\bf n}}
\newcommand{\bx}{{\bf x}}

\def\T{{\mathcal T}}
\def\E{{\mathcal E}}
\def\Q{{\mathbb Q}}

\def\l{{\langle}}
\def\r{{\rangle}}

\def\bn{{\bf n}}
\def\bq{{\bf q}}

\def\bbQ{\mathbb{Q}}
\newcommand{\pT}{{\partial T}}

\def\3bar{{|\hspace{-.02in}|\hspace{-.02in}|}}

  \def\b#1{\mathbf{#1}} 
\def\a#1{\begin{align*}#1\end{align*}} \def\an#1{\begin{align}#1\end{align}} 
 
\def\p#1{\begin{pmatrix}#1\end{pmatrix}} 
  \numberwithin{equation}{section}
\numberwithin{table}{section} \numberwithin{figure}{section}

\title{A new weak gradient for the stabilizer free weak Galerkin method
    with polynomial reduction }

\author{ Xiu Ye\thanks{Department of
Mathematics, University of Arkansas at Little Rock, Little Rock, AR
72204 (xxye@ualr.edu). This research was supported in part by
National Science Foundation Grant DMS-1620016.}
\and
Shangyou Zhang\thanks{Department of
Mathematical Sciences, University of Delaware, Newark, DE 19716, U.S.A. (szhang@udel.edu).}
}

\begin{document}

\maketitle

\begin{abstract}
The weak Galerkin (WG) finite element method is an effective and flexible general numerical technique for solving partial differential equations. It is a natural extension of the classic conforming finite element method for discontinuous approximations, which maintains simple finite element formulation.  Stabilizer free weak Galerkin methods further simplify the WG methods and reduce computational complexity.
This paper explores the possibility of optimal combination of polynomial spaces that minimize the number of unknowns in the stabilizer free WG schemes without compromising the accuracy of the numerical approximation.
A new stabilizer free weak Galerkin finite element method is proposed and analyzed with polynomial degree reduction. To achieve such a goal, a new definition of weak gradient is introduced. Error estimates of optimal order are established for the corresponding WG approximations in both a discrete $H^1$ norm and the standard $L^2$ norm. The numerical examples are tested on various meshes and confirm the theory.
\end{abstract}


\pagestyle{myheadings}

\section{Introduction}\label{Section:Introduction}
The weak Galerkin finite element method is an effective and flexible numerical technique
for solving partial differential equations.
The main idea of weak Galerkin finite element methods is the use of weak functions and their corresponding discrete weak derivatives in algorithm design. The WG method was first introduced in \cite{wy,wymix} and then has been applied to solve various partial differential equations
\cite{hmy,Lin2014,lyzz,mwy-helm,mwy-biha,mwy-soe,mwyz-maxwell, mwyz-interface,mwy-brinkman,Shields,ww-div-curl,wy-stokes}.

A stabilizing/penalty term is often essential in finite element methods with discontinuous approximations to enforce connection of discontinuous functions across element boundaries.
Removing stabilizers from discontinuous finite element methods  simplifies finite element formulations
and reduces programming complexity. Stabilizer free WG finite element methods have been studied in \cite{yzwg, aw, yzz}. The idea is increasing the connectivity of a weak function across element boundary by raising the degree of polynomials for computing weak derivatives. In \cite{yzwg}, we have proved that a stabilizer can be removed from the WG finite element formulation for the WG element $(P_k(T),P_k(e),[P_{j}(T)]^d)$ if $j\ge k+n-1$, where $n$ is the number of edges/faces of an element. The condition $j\ge k+n-1$ has been relaxed in \cite{aw}. Stabilizer free DG methods have also been developed in \cite{yzcdg1,yzcdg2}.

For simplicity, we demonstrate the idea by using the second order elliptic problem that seeks an
unknown function $u$ satisfying
\begin{eqnarray}
-\nabla\cdot (a\nabla u)&=&f\quad \mbox{in}\;\Omega,\label{pde}\\
u&=&g\quad\mbox{on}\;\partial\Omega,\label{bc}
\end{eqnarray}
where $\Omega$ is a polytopal domain in $\mathbb{R}^d$, $\nabla u$ denotes the gradient of
the function $u$, and $a$ is a symmetric
$d\times d$ matrix-valued function in $\Omega$. We shall assume that
there exist two positive numbers $\lambda_1,\lambda_2>0$ such that
\begin{equation}\label{ellipticity}
\lambda_1 \xi^t\xi\le \xi^ta\xi\le \lambda_2 \xi^t\xi,\qquad\forall
\xi\in\mathbb{R}^d.
\end{equation}
Here $\xi$ is understood as a column vector and $\xi^t$ is the transpose
of $\xi$.

The goal of this paper is to propose and analyze a stabilizer free weak Galerkin
method for (\ref{pde})-(\ref{bc}) by using less number of unknowns than that of \cite{yzwg} without compromising the order of
convergence. The WG scheme will use the configuration of $(P_k(T), P_{k-1}(e)$, $P_{j}(T)^d)$. For the WG element $(P_k(T), P_{k-1}(e)$, $P_{j}(T)^d)$, the stabilizer free WG method with the standard definition of weak gradient only produces suboptimal convergence rates in both energy norm and the $L^2$ norm, shown in Table \ref{t111} (\cite{y-z-table}).

\begin{table}[h]
 \begin{center}\caption{Weak gradient calculated by (\ref{d-d-1}), $\3bar\cdot\3bar=O(h^{r_1})$ and $\|\cdot\|=O(h^{r_2})$.}
    \label{t111}
\begin{tabular}{|c|c|c|c|c|}
\hline
$P_{k}(T)$ & $P_{k-1}(e)$ &  $[P_{k+1}(T)]^d$  &  $r_1$  &  $r_2$ \\
\hline
$P_1(T)$ & $P_0(e)$ & $[P_2(T)]^2$ & $0$ & $0$ \\
\hline
$P_2(T)$ & $P_1(e)$ & $[P_3(T)]^2$ &  $1$ & $2$  \\
\hline
$P_3(T)$ & $P_2(e)$ & $[P_4(T)]^2$ & $2$ & $3$  \\
\hline
\end{tabular}
\end{center}
\end{table}

The standard definition for a weak gradient  $\nabla_wv$ of a weak function  $v=\{v_0,v_b\}$  is a piecewise vector valued polynomial such that on each $T\in\T_h$,  $\nabla_w v \in [P_{k+1}(T)]^2$ satisfies (\cite{wy,wymix, mwy-soe})
\begin{equation}\label{d-d-1}
  (\nabla_w v, \bq)_T = -(v_0, \nabla\cdot \bq)_T+ \langle v_b, \bq\cdot\bn\rangle_{\partial T}\qquad
   \forall \bq\in [P_{j}(T)]^2.
\end{equation}
A new way of defining weak gradient is introduced in (\ref{d-d}) for the WG element $(P_{k}(T),P_{k-1}(e),[P_{j}(T)]^2)$ such that the our new corresponding stabilizer free WG approximation converges to the true solution with optimal order convergence rates, shown in Table \ref{t222}.

\begin{table}[h]
 \begin{center}
\caption{Weak gradient calculated by (\ref{d-d}), $\3bar\cdot\3bar=O(h^{r_1})$ and $\|\cdot\|=O(h^{r_2})$.}
    \label{t222}
\begin{tabular}{|c|c|c|c|c|}
\hline
$P_{k}(T)$ & $P_{k-1}(e)$ &  $[P_{k+1}(T)]^d$  &  $r_1$  &  $r_2$ \\
\hline
$P_1(T)$ & $P_0(e)$ & $[P_2(T)]^2$ & $1$ & $2$ \\
\hline
$P_2(T)$ & $P_1(e)$ & $[P_3(T)]^2$ &  $2$ & $3$  \\
\hline
$P_3(T)$ & $P_2(e)$ & $[P_4(T)]^2$ & $3$ & $4$  \\
\hline
\end{tabular}
\end{center}
\end{table}

We also prove the optimal convergence rates theoretically for the stabilizer free WG approximation in an energy norm  and in the $L^2$ norm. The numerical examples are tested on different finite element partitions.

\section{Weak Galerkin Finite Element Schemes}

Let ${\cal T}_h$ be a partition of the domain $\Omega$ consisting of
polygons in two dimension or polyhedra in three dimension satisfying
a set of conditions specified in \cite{wymix}. Denote by
${\cal E}_h$ the set of all edges/faces in ${\cal T}_h$, and let
${\cal E}_h^0={\cal E}_h\backslash\partial\Omega$ be the set of all
interior edges/faces. For simplicity, we will use term edge for edge/face without confusion.

For a given integer $k \ge 1$, let $V_h$ be the weak Galerkin finite
element space associated with $\T_h$ defined as follows
\begin{equation}\label{vhspace}
V_h=\{v=\{v_0,v_b\}:\; v_0|_T\in P_k(T),\ v_b|_e\in P_{k-1}(e),\ e\subset\pT,  T\in \T_h\}
\end{equation}
and its subspace $V_h^0$ is defined as
\begin{equation}\label{vh0space}
V^0_h=\{v: \ v\in V_h,\  v_b=0 \mbox{ on } \partial\Omega\}.
\end{equation}
We would like to emphasize that any function $v\in V_h$ has a single
value $v_b$ on each edge $e\in\E_h$.

For $v=\{v_0,v_b\}\in V_h\cup H^1(\Omega)$, a weak gradient $\nabla_wv$ is a piecewise vector valued polynomial such that on each $T\in\T_h$,  $\nabla_w v |_T \in [P_{j}(T)]^d$ satisfies
\begin{equation}\label{d-d}
  (\nabla_w v, \bq)_T = (\nabla v_0,  \bq)_T+ \langle Q_b(v_b- v_0), \bq\cdot\bn\rangle_{\partial T}\qquad
   \forall \bq\in [P_{j}(T)]^d,
\end{equation}
where $j>k$ depends on the shape of the elements and will be determined later.
In the above equation, we let $v_0=v$ and $v_b=v$ if $v\in H^1(\Omega)$.

For simplicity, we adopt the following notations,
\begin{eqnarray*}
(v,w)_{\T_h} &=&\sum_{T\in\T_h}(v,w)_T=\sum_{T\in\T_h}\int_T vw d\bx,\\
 \l v,w\r_{\partial\T_h}&=&\sum_{T\in\T_h} \l v,w\r_\pT=\sum_{T\in\T_h} \int_\pT vw ds.
\end{eqnarray*}

Let $Q_0$ and $Q_b$ be the two element-wise defined $L^2$ projections onto $P_k(T)$ and $P_{k-1}(e)$  on each $T\in\T_h$, respectively. Define $Q_hu=\{Q_0u,Q_bu\}\in V_h$. Let $\Q_h$ be the element-wise defined $L^2$ projection onto $[P_{j}(T)]^d$ on each element $T\in\T_h$.

\begin{algorithm}
A numerical approximation for (\ref{pde})-(\ref{bc}) can be
obtained by seeking $u_h=\{u_0,u_b\}\in V_h$
satisfying $u_b=Q_bg$ on $\partial\Omega$ and  the following equation:
\begin{equation}\label{wg}
(a\nabla_wu_h,\nabla_wv)_{\T_h}=(f,\; v_0) \quad\forall v=\{v_0,v_b\}\in V_h^0.
\end{equation}
\end{algorithm}

The following lemma reveals a nice property of the weak gradient.
\begin{lemma}
Let $\phi\in H^1(\Omega)$, then on any $T\in\T_h$, we have
\begin{eqnarray}
\nabla_w \phi &=&\Q_h\nabla\phi.\label{key1}
\end{eqnarray}
\end{lemma}
\begin{proof}
The definition of weak gradient (\ref{d-d}) implies that for
any $\bq\in [P_{j}(T)]^d$
\begin{eqnarray*}
(\nabla_w \phi,\bq)_T &=& (\nabla\phi,\bq)_T
+\langle Q_b(\phi-\phi),\bq\cdot\bn\rangle_{\pT}\\
&=& (\Q_h\nabla\phi,\bq)_T,
\end{eqnarray*}
which implies (\ref{key1}).
\end{proof}

For any function $\varphi\in H^1(T)$, the following trace
inequality holds true (see \cite{wymix} for details):
\begin{equation}\label{trace}
\|\varphi\|_{e}^2 \leq C \left( h_T^{-1} \|\varphi\|_T^2 + h_T
\|\nabla \varphi\|_{T}^2\right).
\end{equation}

\section{Well Posedness}

For any $v\in V_h \cup H^1(\Omega)$, define two semi-norms
\begin{eqnarray}
\3bar v\3bar^2_1&=&(\nabla_wv,\nabla_wv)_{\T_h},\label{3barnorm}\\
\3bar v\3bar^2&=&(a\nabla_wv,\nabla_wv)_{\T_h}.\label{3barnorm-1}
\end{eqnarray}
It follows from (\ref{ellipticity}) that there exist two positive constants $\alpha$ and $\beta$ such that
\begin{eqnarray}
\alpha\3bar v\3bar\le \3bar v\3bar_1\le \beta\3bar v\3bar.\label{norm-equ}
\end{eqnarray}

We introduce a discrete $H^1$ semi-norm as follows:
\begin{equation}\label{norm}
\|v\|_{1,h} = \left( \sum_{T\in\T_h}\left(\|\nabla
v_0\|_T^2+h_T^{-1} \|Q_b (v_0-v_b)\|^2_\pT\right) \right)^{\frac12}.
\end{equation}
It is easy to see that $\|v\|_{1,h}$ defines a norm in $V_h^0$.

Next we will show that $\3bar  \cdot \3bar$ also defines a norm for $V_h^0$ by proving
the equivalence of $\3bar\cdot\3bar$ and $\|\cdot\|_{1,h}$ in $V_h$. First we need the following lemma.

\medskip

\begin{lemma}\label{zhang} (\cite{yzcdg2})
Let $T$ be a convex, shape-regular $n$-polygon/polyhedron of size $h_T$.
Let $e\in\pT$ be an edge/face-polygon of $T$, of size $Ch_T$. Let $v_h\in V_h$ and
  $v_h=\{v_0,v_b\}$ on $T$.  Then there exists a polynomial $\bq\in [P_{j}(T)]^d$, $j=n+k-1$,
   such that
\an{ \label{q1} -(\nabla v_0,\b q)_T &= 0, \\
  \label{q2}
   \l Q_b(v_0-v_b), \b q\cdot \b n\r_e & = \|Q_b(v_0-v_b)\|_e^2\quad \forall e\subset \pT, \\
   \|\b q\|_T^2 & \le  C  h_T \|Q_b(v_0-v_b)\|_e ^2. \label{q3} }
\end{lemma}

\begin{lemma} There exist two positive constants $C_1$ and $C_2$ such
that for any $v=\{v_0,v_b\}\in V_h$, we have
\begin{equation}\label{happy}
C_1 \|v\|_{1,h}\le \3bar v\3bar \leq C_2 \|v\|_{1,h}.
\end{equation}
\end{lemma}

\medskip

\begin{proof}
For any $v=\{v_0,v_b\}\in V_h$, it follows from the definition of
weak gradient (\ref{d-d}) and integration by parts that on each $T\in\T_h$
\begin{eqnarray}\label{n-1}
(\nabla_wv,\bq)_T=(\nabla v_0,\bq)_T+\l Q_b(v_b-Qv_0),
\bq\cdot\bn\r_\pT\quad \forall \bq\in [P_{j}(T)]^d.
\end{eqnarray}
By letting $\bq=\nabla_w v|_T$ in (\ref{n-1}) we arrive at
\begin{eqnarray*}
(\nabla_wv,\nabla_w v)_T=(\nabla v_0,\nabla_w v)_T+\l Q_b(v_b-v_0),
\nabla_w v\cdot\bn\r_\pT.
\end{eqnarray*}
Letting $\bq=\nabla v_0|_T$ in (\ref{n-1}) implies
\begin{eqnarray}
(\nabla_wv,\nabla v_0)_T=(\nabla v_0,\nabla v_0)_T+\l Q_b(v_b-v_0),
\nabla v_0\cdot\bn\r_\pT.\label{n40}
\end{eqnarray}
From the trace inequality (\ref{trace}) and the inverse inequality
we have
\begin{eqnarray*}
\|\nabla_wv\|^2_T &\le& \|\nabla v_0\|_T \|\nabla_w v\|_T+ \|
Q_b(v_0-v_b)\|_\pT \|\nabla_w v\|_\pT\\
&\le& \|\nabla v_0\|_T \|\nabla_w v\|_T+ Ch_T^{-1/2}\|
Q_b(v_0-v_b)\|_\pT \|\nabla_w v\|_T,
\end{eqnarray*}
which implies
$$
\|\nabla_w v\|_T \le C \left(\|\nabla v_0\|_T +h_T^{-1/2}\|Q_b(v_0-v_b)\|_\pT\right),
$$
and consequently
$$\3bar v\3bar \leq C_2 \|v\|_{1,h}.$$

Next we will prove $C_1 \|v\|_{1,h}\le \3bar v\3bar $. First we need to prove
\begin{equation}\label{nnn1}
h_e^{-1/2}\|Q_b(v_0-v_b)\|_e\le C\|\nabla_w v\|_T.
\end{equation}
For $e\in\E_h$ and $T\in\T_h$ with $e\subset \pT$,
it has been proved in Lemma \ref{zhang} that there exists $\bq_0\in [P_{j}(T)]^d$ such that
\begin{equation}\label{2e}
(\nabla v_0,\bq_0)_T=0, \
    \ \ \l Q_b(v_b-v_0), \bq_0\cdot\bn\r_\pT =\|Q_b(v_0-v_b)\|_{\pT}^2,
\end{equation}
and
\begin{equation}
\|\bq_0\|_T \le C h_T^{1/2} \| Q_b(v_b-v_0)\|_\pT.\label{22e}
\end{equation}
Substituting  $\bq_0$ into (\ref{n-1}), we get
\begin{equation}\label{n3}
(\nabla_wv,\bq_0)_T=\|Q_b(v_b-v_0)\|^2_\pT.
\end{equation}
It follows from Cauchy-Schwarz inequality and (\ref{22e}) that
\[
\|Q_b(v_b-v_0)\|^2_\pT\le C\|\nabla_w v\|_T\|\bq_0\|_T
 \le Ch_T^{1/2}\|\nabla_w v\|_T\|Q_b(v_0-v_b)\|_\pT,
\]
which implies
\begin{equation}\label{n4}
h_T^{-1/2}\|Q_b(v_0-v_b)\|_\pT\le C\|\nabla_w v\|_T.
\end{equation}
It follows from (\ref{n40}),  the trace inequality, the inverse inequality and (\ref{n4}),
\a{
\|\nabla v_0\|_T^2 & \leq \|\nabla_w v\|_T \|\nabla v_0\|_T
+Ch_T^{-1/2}\| Q_b(v_0-v_b)\|_\pT \|\nabla v_0\|_T
  \\ & \le C\|\nabla_w v\|_T \|\nabla v_0\|_T,
}
which implies
$$
\|\nabla v_0\|_T \leq C\|\nabla_w v\|_T.
$$
Combining the above estimate and (\ref{n4}),
 we prove  the lower bound of (\ref{happy}) and complete the proof of the lemma.
\end{proof}

\medskip

\begin{lemma}
The weak Galerkin finite element scheme (\ref{wg}) has a unique
solution.
\end{lemma}

\smallskip

\begin{proof}
If $u_h^{(1)}$ and $u_h^{(2)}$ are two solutions of (\ref{wg}), then
$\varepsilon_h=u_h^{(1)}-u_h^{(2)}\in V_h^0$ would satisfy the following equation
$$
(a\nabla_w \varepsilon_h,\nabla_w v)_{\T_h}=0,\qquad\forall v\in V_h^0.
$$
 Then by letting $v=\varepsilon_h$ in the above
equation and (\ref{norm-equ}), we arrive at
$$
\3bar \varepsilon_h\3bar^2 =(a\nabla_w \varepsilon_h,\nabla_w \varepsilon_h)=0.
$$
It follows from (\ref{happy}) that $\|\varepsilon_h\|_{1,h}=0$. Since $\|\cdot\|_{1,h}$ is a norm in $V_h^0$, one has $\varepsilon_h=0$.
 This completes the proof of the lemma.
\end{proof}

\section{Error Analysis}

The goal of this section is to establish  error estimates for
the weak Galerkin finite element solution $u_h$ arising from (\ref{wg}).
For simplicity of analysis, we assume that the coefficient tensor
$a$ in (\ref{pde}) is a piecewise constant matrix with respect to
the finite element partition $\T_h$. The result can be extended to
variable tensors without any difficulty, provided that the tensor
$a$ is piecewise sufficiently smooth.

\subsection{Error Equation}

Let $e_h=u-u_h$ and $\epsilon_h=Q_hu-u_h\in V_h$. In this section, we derive an error equation that $e_h$ satisfies.

\begin{lemma}
For any $v\in V_h^0$, the following error equation holds true
\begin{eqnarray}
(a\nabla_we_h,\nabla_wv)_{\T_h}=e_1(u,v)+e_2(u,v),\label{ee}
\end{eqnarray}
where
\begin{eqnarray*}
e_1(u,v)&=& \langle a(\nabla u-\Q_h\nabla u)\cdot\bn,\;Q_bv_0-v_b\rangle_{\partial\T_h},\\
e_2(u,v)&=& \l a\nabla u \cdot\bn, v_0-Q_bv_0\rangle_{\partial\T_h}.
\end{eqnarray*}
\end{lemma}

\begin{proof}
For $v=\{v_0,v_b\}\in V_h^0$, testing (\ref{pde}) by  $v_0$  and using the fact that
$\langle a\nabla u\cdot\bn, v_b\rangle_{\partial\T_h}=0$,  we arrive at
\begin{equation}\label{m0}
(a\nabla u,\nabla v_0)_{\T_h}- \langle
a\nabla u\cdot\bn,v_0-v_b\rangle_{\partial\T_h}=(f,v_0).
\end{equation}
Obviously, we have
\begin{eqnarray}
\langle
a\nabla u\cdot\bn,v_0-v_b\rangle_{\partial\T_h}&=&\langle a\nabla u\cdot\bn, Q_bv_0-v_b\rangle_{\partial\T_h}
+\langle a\nabla u\cdot\bn,v_0-Q_bv_0\rangle_{\partial\T_h}.\label{j0}
\end{eqnarray}
Combining (\ref{m0}) and (\ref{j0}) gives
\begin{equation}\label{m1}
(a\nabla u,\nabla v_0)_{\T_h}- \langle
a\nabla u\cdot\bn,Q_bv_0-v_b\rangle_{\partial\T_h}=(f,v_0)+e_2(u,v).
\end{equation}

It follows from (\ref{d-d}) and (\ref{key1})  that
\begin{eqnarray}
(a\nabla u,\nabla v_0)_{\T_h}&=&(a\Q_h\nabla  u,\nabla v_0)_{\T_h}\nonumber\\
&=&(a\Q_h\nabla u, \nabla_w v)_{\T_h}+\langle Q_bv_0-v_b,a\Q_h\nabla u\cdot\bn\rangle_{\partial\T_h}\nonumber\\
&=&(a\nabla_w u, \nabla_w v)_{\T_h}+\langle Q_bv_0-v_b,a\Q_h\nabla u\cdot\bn\rangle_{\partial\T_h}.\label{j1}
\end{eqnarray}
Using (\ref{m1}) and (\ref{j1}), we have
\begin{eqnarray}
(a\nabla_w u,\nabla_w v)_{\T_h}
&=&(f,v_0)+e_1(u,v)+e_2(u,v).\label{j2}
\end{eqnarray}
Subtracting (\ref{wg}) from (\ref{j2}) yields,
\begin{eqnarray*}
(a\nabla_we_h,\nabla_wv)_{\T_h}=e_1(u,v)+e_2(u,v)\quad \forall v\in V_h^0.
\end{eqnarray*}
This completes the proof of the lemma.
\end{proof}

\subsection{Error Estimates in Energy Norm}

Optimal convergence rate of the WG approximation in energy norm will be obtained in this section.
First we will bound the two terms $e_1(u,v)$ and $e_2(u,v)$ in the error equation (\ref{ee}).

\begin{lemma} For any $w\in H^{k+1}(\Omega)$ and
$v=\{v_0,v_b\}\in V_h^0$, we have
\begin{eqnarray}
|e_1(w, v)|&\le&Ch^{k}|w|_{k+1}\3bar v\3bar,\label{mmm1}\\
|e_2(w, v)|&\le&Ch^{k}|w|_{k+1}\3bar v\3bar.\label{mmm4}
\end{eqnarray}
\end{lemma}

\medskip

\begin{proof}
Using the Cauchy-Schwarz inequality, the trace inequality (\ref{trace}), (\ref{ellipticity}) and (\ref{happy}), we have
\begin{eqnarray*}
|e_1(w,v)|&=&\left|\sum_{T\in\T_h}\langle a(\nabla w-\Q_h\nabla
w)\cdot\bn, Q_bv_0-v_b\rangle_\pT\right|\\
&\le & C \sum_{T\in\T_h}\|\nabla w-\Q_h\nabla w\|_{\pT}
\|Q_bv_0-v_b\|_\pT\nonumber\\
&\le & C \left(\sum_{T\in\T_h}h_T\|(\nabla w-\Q_h\nabla w)\|_{\pT}^2\right)^{\frac12}
\left(\sum_{T\in\T_h}h_T^{-1}\|Q_bv_0-v_b\|_\pT^2\right)^{\frac12}\\
&\le & Ch^{k}|w|_{k+1}\3bar v\3bar.
\end{eqnarray*}

Let $\Q_{k-1}$ be the element-wise defined $L^2$ projection onto $[P_{k-1}(T)]^d$ on each $T\in\T_h$.
Using the Cauchy-Schwarz inequality, the trace inequality (\ref{trace}), (\ref{ellipticity}) and the inverse inequality, we have
\begin{eqnarray*}
|e_2(w,v)|&=&\left|\sum_{T\in\T_h}\langle a\nabla w\cdot\bn, v_0-Q_bv_0\rangle_\pT\right|\\
&=&\left|\sum_{T\in\T_h}\langle a(\nabla w-\Q_{k-1}\nabla w)\cdot\bn, v_0-Q_bv_0\rangle_\pT\right|\\
&\le & C \sum_{T\in\T_h}\|\nabla w-\Q_{k-1}\nabla w\|_{\pT}
\|v_0-Q_bv_0\|_\pT\nonumber\\
&\le & C \left(\sum_{T\in\T_h}h_T\|(\nabla w-\Q_{k-1}\nabla w)\|_{\pT}^2\right)^{\frac12}
\left(\sum_{T\in\T_h}h_T^{-1}h_T^2\|\nabla v_0\|_\pT^2\right)^{\frac12}\\
&\le & Ch^{k}|w|_{k+1}\3bar v\3bar.
\end{eqnarray*}
We have proved the lemma.
\end{proof}

\begin{lemma}
Let $u\in H^{k+1}(\Omega)$, then
\begin{equation}\label{eee2}
\3baru-Q_hu\3bar\le Ch^k|u|_{k+1}.
\end{equation}
\end{lemma}
\begin{proof}
It follows from (\ref{d-d}) and  (\ref{trace}),
\begin{eqnarray*}
|(\nabla_w(u-Q_hu), \bq)_{T}|&=&|(\nabla(u-Q_0u), \bq)_{T}+\l Q_b(Q_0u-Q_bu), \bq\cdot\bn\r_{\pT}|\\
&=&|(\nabla(u-Q_0u), \bq)_{T}+\l Q_b(Q_0u-u), \bq\cdot\bn\r_{\pT}|\\
&\le& \|\nabla (u-Q_0u)\|_T\|\bq\|_T+Ch^{-1/2}\|Q_0u-u\|_\pT\|\bq\|_T\\
&\le& Ch^k|u|_{k+1, T}\|\bq\|_T.
\end{eqnarray*}
Letting $\bq=\nabla_w(u-Q_hu)$ in the above equation and taking summation over $T$, we have
\[
\3baru-Q_hu\3bar\le Ch^k|u|_{k+1}.
\]
We have proved the lemma.
\end{proof}

\begin{theorem} Let $u_h\in V_h$ be the weak Galerkin finite element solution of (\ref{wg}). Assume the exact solution $u\in H^{k+1}(\Omega)$. Then,
there exists a constant $C$ such that
\begin{equation}\label{err1}
\3bar u-u_h\3bar \le Ch^{k}|u|_{k+1}.
\end{equation}
\end{theorem}
\begin{proof}
It is straightforward to obtain
\begin{eqnarray}
 \3bar e_h\3bar^2&=&(a\nabla_we_h, \nabla_we_h)_{\T_h}\label{eee1}\\
&=&(a(\nabla_wu-\nabla_wu_h),\nabla_we_h)_{\T_h}\nonumber\\
&=&(a(\nabla_wQ_hu-\nabla_wu_h),\nabla_we_h)_{\T_h}+(a(\nabla_wu-\nabla_wQ_hu),\nabla_we_h)_{\T_h}\nonumber\\
&=&(a\nabla_we_h,\nabla_w\epsilon_h)_{\T_h}+(a\nabla_w(u-Q_hu),\nabla_we_h)_{\T_h}.\nonumber
\end{eqnarray}
We will bound the two terms in (\ref{eee1}).
Letting $v=\epsilon_h\in V_h$ in (\ref{ee})  and using (\ref{mmm1}), (\ref{mmm4}) and (\ref{eee2}), we have
\begin{eqnarray}
|(a\nabla_we_h,\nabla_w\epsilon_h)_{\T_h}|&=&|e_1(u,\epsilon_h)+e_2(u,\epsilon_h)|\nonumber\\
&\le& Ch^{k}|u|_{k+1}\3bar \epsilon_h\3bar\nonumber\\
&\le& Ch^{k}|u|_{k+1}\3bar Q_hu-u_h\3bar\nonumber\\
&\le& Ch^{k}|u|_{k+1}(\3bar Q_hu-u\3bar+\3bar u-u_h\3bar)\nonumber\\
&\le& Ch^{2k}|u|^2_{k+1}+\frac14 \3bare_h\3bar^2.\label{eee3}
\end{eqnarray}
The estimate (\ref{eee2}) implies
\begin{eqnarray}
|(\nabla_w(u-Q_hu),\nabla_we_h)_{\T_h}|&\le& C\3bar u-Q_hu\3bar \3bar e_h\3bar\nonumber\\
&\le& Ch^{2k}|u|^2_{k+1}+\frac14\3bar e_h\3bar^2.\label{eee4}
\end{eqnarray}
Combining the estimates (\ref{eee3}) and  (\ref{eee4}) with (\ref{eee1}), we arrive
\[
\3bar e_h\3bar \le Ch^{k}|u|_{k+1},
\]
which completes the proof.
\end{proof}

The estimates (\ref{eee2}) and (\ref{err1}) imply
\begin{equation}\label{err10}
\3bar Q_hu-u_h\3bar \le Ch^{k}|u|_{k+1}.
\end{equation}

\subsection{Error Estimates in $L^2$ Norm}

The standard duality argument is used to obtain $L^2$ error estimate.
Recall $e_h=\{e_0,e_b\}=u-u_h$ and $\epsilon_h=\{\epsilon_0,\epsilon_b\}=Q_hu-u_h$.
The dual problem seeks $\Phi\in H_0^1(\Omega)$ satisfying
\begin{eqnarray}
-\nabla\cdot a\nabla\Phi&=& \epsilon_0\quad
\mbox{in}\;\Omega.\label{dual}
\end{eqnarray}
Assume that the following $H^{2}$-regularity holds
\begin{equation}\label{reg}
\|\Phi\|_2\le C\|\epsilon_0\|.
\end{equation}

\begin{theorem} Let $u_h\in V_h$ be the weak Galerkin finite element solution of (\ref{wg}). Assume that the
exact solution $u\in H^{k+1}(\Omega)$ and (\ref{reg}) holds true.
 Then, there exists a constant $C$ such that
\begin{equation}\label{err2}
\|Q_0u-u_0\| \le Ch^{k+1}|u|_{k+1}.
\end{equation}
\end{theorem}

\begin{proof}
By testing (\ref{dual}) with $\epsilon_0$ we obtain
\begin{eqnarray}\nonumber
\|\epsilon_0\|^2&=&-(\nabla\cdot (a\nabla\Phi),\epsilon_0)\\
&=&(a\nabla \Phi,\ \nabla \epsilon_0)_{\T_h}-\l
a\nabla\Phi\cdot\bn,\ \epsilon_0-\epsilon_b\r_{\partial\T_h}\nonumber\\
&=&(a\nabla \Phi,\ \nabla \epsilon_0)_{\T_h}-\l a\nabla\Phi\cdot\bn,\ Q_b \epsilon_0- \epsilon_b\r_{\partial\T_h}-\l a\nabla\Phi\cdot\bn,\ \epsilon_0- Q_b\epsilon_0\r_{\partial\T_h}\nonumber\\
&=&(a\nabla \Phi,\ \nabla \epsilon_0)_{\T_h}-\l
a\nabla\Phi\cdot\bn,\ Q_b \epsilon_0- \epsilon_b\r_{\partial\T_h}-e_2(\Phi,\epsilon_h),\label{jw.08}
\end{eqnarray}
where we have used the fact that $\epsilon_b=0$ on $\partial\Omega$.
Setting $u=\Phi$ and $v=\epsilon_h$ in (\ref{j1}) yields
\begin{eqnarray}
(a\nabla\Phi,\;\nabla \epsilon_0)_{\T_h}&=&(a \nabla_w \Phi, \nabla_w \epsilon_h)_{\T_h}+\langle Q_b\epsilon_0-\epsilon_b,a\Q_h\nabla \Phi\cdot\bn\rangle_{\partial\T_h}.\label{j1-new}
\end{eqnarray}
Substituting (\ref{j1-new}) into (\ref{jw.08}) gives
\begin{eqnarray}
\|\epsilon_0\|^2&=&(a\nabla_w \epsilon_h,\ \nabla_w \Phi)_{\T_h}-e_1(\Phi, \epsilon_h)-e_2(\Phi, \epsilon_h)\nonumber\\
&=&(a\nabla_w e_h,\ \nabla_w \Phi)_{\T_h}+(a\nabla_w (Q_hu-u),\ \nabla_w \Phi)_{\T_h}-e_1(\Phi, \epsilon_h)-e_2(\Phi, \epsilon_h)\nonumber\\
&=&(a\nabla_w e_h,\ \nabla_w Q_h\Phi)_{\T_h}+(a\nabla_w e_h,\ \nabla_w (\Phi-Q_h\Phi))_{\T_h}\nonumber\\
&+&(a\nabla_w (Q_hu-u),\ \nabla_w \Phi)_{\T_h}-e_1(\Phi, \epsilon_h)-e_2(\Phi, \epsilon_h)\nonumber\\
&=&e_1(u,Q_h\Phi)+e_2(u,Q_h\Phi)-e_1(\Phi,\epsilon_h)-e_2(\Phi,\epsilon_h)\nonumber\\
&+&(a\nabla_w e_h,\ \nabla_w (\Phi-Q_h\Phi))_{\T_h}+(a\nabla_w (Q_hu-u),\ \nabla_w \Phi)_{\T_h}.\label{m2}
\end{eqnarray}

Let us bound all the terms on the right hand side of (\ref{m2}) one by
one. Using the Cauchy-Schwarz  inequality and the definition of $Q_b$, we obtain
\begin{eqnarray}
|e_1(u,Q_h\Phi)|&=&\left|\sum_{T\in\T_h} \langle a(\nabla u-\bbQ_h\nabla
u)\cdot\bn,\; Q_bQ_0\Phi-Q_b\Phi\rangle_\pT \right|\nonumber\\
&=&\left|\sum_{T\in\T_h} \langle a(\nabla u-\bbQ_h\nabla
u)\cdot\bn,\; Q_b(Q_0\Phi-\Phi)\rangle_\pT \right|\nonumber\\
&\le&C\sum_{T\in\T_h} \|\nabla u-\bbQ_h\nabla u\|_{\pT}\| Q_0\Phi-\Phi\|_\pT \nonumber\\
&\le& C\left(\sum_{T\in\T_h}\|\nabla u-\bbQ_h\nabla u\|^2_\pT\right)^{1/2}\left(\sum_{T\in\T_h}\|Q_0\Phi-\Phi\|^2_\pT\right)^{1/2}\label{1st-term}
\end{eqnarray}
From the trace inequality (\ref{trace}) we have
$$
\left(\sum_{T\in\T_h}\|Q_0\Phi-\Phi\|^2_\pT\right)^{1/2} \leq C
h^{\frac32}\|\Phi\|_2
$$
and
$$
\left(\sum_{T\in\T_h}\|a(\nabla u-\bbQ_h\nabla
u)\|^2_\pT\right)^{1/2} \leq Ch^{k-\frac12}\|u\|_{k+1}.
$$
Combining  the above two estimates with (\ref{1st-term}) gives
\begin{eqnarray}\label{1st-term-complete}
|e_1(u,Q_h\Phi)| \leq C h^{k+1} \|u\|_{k+1}
\|\Phi\|_2.
\end{eqnarray}

Using the Cauchy-Schwarz inequality, the trace inequality (\ref{trace}) and (\ref{ellipticity}), we have
\begin{eqnarray*}
|e_2(u,Q_h\Phi)|&=&\left|\sum_{T\in\T_h}\langle a\nabla u\cdot\bn, Q_0\Phi-Q_bQ_0\Phi\rangle_\pT\right|\\
&=&\left|\sum_{T\in\T_h}\langle a(\nabla u-\Q_{k-1}\nabla u)\cdot\bn,  Q_0\Phi-Q_bQ_0\Phi\rangle_\pT\right|\\
&\le & C \sum_{T\in\T_h}\|\nabla u-\Q_{k-1}\nabla u\|_{\pT}
\| Q_0\Phi-Q_bQ_0\Phi\|_\pT\nonumber\\
&\le & Ch^{k-1/2}|u|_{k+1}h^{3/2}\|\Phi\|_2\\
&\le & Ch^{k+1}|u|_{k+1}\|\Phi\|_2.
\end{eqnarray*}

It follows from (\ref{mmm1}), (\ref{mmm4}) and (\ref{err10}) that
\begin{eqnarray*}
|e_1(\Phi, \epsilon_h)|&\le&Ch|\Phi|_{2}\3bar \epsilon_h\3bar\le Ch^{k+1}|u|_{k+1}\|\Phi\|_2\label{t30},
\end{eqnarray*}
and
\begin{eqnarray*}
|e_2(\Phi, \epsilon_h)|&\le&Ch|\Phi|_{2}\3bar \epsilon _h\3bar\le Ch^{k+1}|u|_{k+1}\|\Phi\|_2\label{t31}.
\end{eqnarray*}
The estimates (\ref{eee2}) and (\ref{err1}) imply
\begin{eqnarray*}
|(a\nabla_w e_h,\ \nabla_w (\Phi-Q_h\Phi))_{\T_h}|&\le& C\3bar e_h\3bar \3bar \Phi-Q_h\Phi\3bar\le Ch^{k+1}|u|_{k+1}\|\Phi\|_2.
\end{eqnarray*}
To bound the term $(a\nabla_w (Q_hu-u),\ \nabla_w \Phi)_{\T_h}$, we define a $L^2$ projection element-wise onto $[P_0(T)]^d$ denoted by $\Q_0$. Then it follows from the definition of weak gradient (\ref{d-d}) and integration by parts,
\begin{eqnarray*}
&& (\nabla_w (Q_hu-u),\ a\Q_0\nabla\Phi)_{T}\\
&=&(\nabla (Q_0u-u),a\Q_0\nabla\Phi)_T+ \l (Q_b(Q_bu-u-(Q_0u-u)), \Q_0\nabla\Phi\cdot\bn\r_\pT\\
&=&(\nabla (Q_0u-u),a\Q_0\nabla\Phi)_T+\l Q_bu-Q_0u, a\Q_0\nabla\Phi\cdot\bn\r_\pT\\
&=&-(Q_0u-u,\nabla\cdot a\Q_0\nabla\Phi)_T+ \l Q_0u-u+Q_bu-Q_0u, a\Q_0\nabla\Phi\cdot\bn\r_\pT\\
&=&\l Q_bu-u, a\Q_0\nabla_w\Phi\cdot\bn\r_\pT=0.
\end{eqnarray*}
Using the equation above, (\ref{key1}) and (\ref{eee2}) and the definition of $\Q_0$, we have
\begin{eqnarray*}
|(a\nabla_w (Q_hu-u),\ \nabla_w \Phi)_{\T_h}|&=&|(a\nabla_w (Q_hu-u),\ \Q_h\nabla \Phi)_{\T_h}|\\
&=&|(\nabla_w (Q_hu-u),\ a\nabla\Phi)_{\T_h}|\\
&=&|(\nabla_w (Q_hu-u),\ a(\nabla\Phi-\Q_0\nabla\Phi))_{\T_h}|\\
&\le& Ch^{k+1}|u|_{k+1}\|\Phi\|_2.
\end{eqnarray*}

Combining all the estimates above with (\ref{m2})  yields
$$
\|\epsilon_0\|^2 \leq C h^{k+1}|u|_{k+1} \|\Phi\|_2.
$$
Using the regularity assumption (\ref{reg}) and the estimate above, we derived (\ref{err2}).
\end{proof}

\section{Numerical Experiments}\label{Section:numerical-experiments}

\subsection{Example 1}
Consider problem (\ref{pde}) with $\Omega=(0,1)^2$ and $a=\p{1&0\\0&1}$.
The source term $f$ and the boundary value $g$ are chosen so that the exact solution is
\an{\label{s1}
    u(x,y)=\sin(x)\sin(\pi y). }
We use uniform triangular meshes as shown in Figure \ref{g-triangle}.
The error and the order of convergence are listed in Table \ref{t1},
   where we have optimal order of convergence for $k\ge 1$ in both $L^2$ norm and
    $H^1$-like triple-bar norm.

\begin{figure}[h!]
 \begin{center} \setlength\unitlength{1.25pt}
\begin{picture}(260,80)(0,0)
  \def\tr{\begin{picture}(20,20)(0,0)\put(0,0){\line(1,0){20}}\put(0,20){\line(1,0){20}}
          \put(0,0){\line(0,1){20}} \put(20,0){\line(0,1){20}} \put(20,0){\line(-1,1){20}}
   \end{picture}}
 {\setlength\unitlength{5pt}
 \multiput(0,0)(20,0){1}{\multiput(0,0)(0,20){1}{\tr}}}

  {\setlength\unitlength{2.5pt}
 \multiput(45,0)(20,0){2}{\multiput(0,0)(0,20){2}{\tr}}}

  \multiput(180,0)(20,0){4}{\multiput(0,0)(0,20){4}{\tr}}

 \end{picture}\end{center}
\caption{Example 1.  The first three triangular grids.}
\label{g-triangle}
\end{figure}
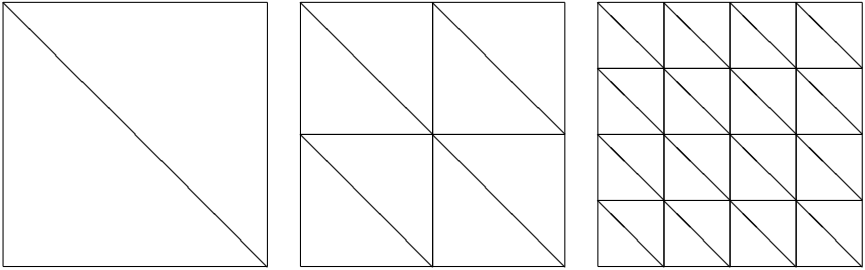

\begin{table}[h!]
    \caption{Example 1. The $P_k-P_{k-1}-[P_{k+1}]^2$ element,
      on triangular grids shown in Figure \ref{g-triangle}.}
    \label{t1}
    \begin{center}
    \begin{tabular}{c|c|cc|cc}  
        \hline
$k$&$\T_l$&$\3bar Q_hu-u_h\3bar$&Rate & $\|Q_hu-u_h\|$ & Rate \\
        \hline
 &6&   0.3871E-01 & 1.00&   0.3306E-03 & 1.99 \\
1&7&   0.1937E-01 & 1.00&   0.8279E-04 & 2.00 \\
 &8&   0.9685E-02 & 1.00&   0.2070E-04 & 2.00 \\
\hline
  & 6&   0.4131E-03 & 1.98&   0.1783E-05 & 2.95 \\
2 & 7&   0.1038E-03 & 1.99&   0.2268E-06 & 2.97 \\
  & 8&   0.2602E-04 & 2.00&   0.2859E-07 & 2.99 \\
\hline
  & 5&   0.2925E-04 & 2.99&   0.1515E-06 & 3.98 \\
3 & 6&   0.3665E-05 & 3.00&   0.9518E-08 & 3.99 \\
  & 7&   0.4587E-06 & 3.00&   0.5963E-09 & 4.00 \\
\hline
   & 5&   0.4091E-06 & 3.99&   0.1592E-08 & 4.97 \\
4 & 6&   0.2568E-07 & 3.99&   0.5026E-10 & 4.99 \\
  & 7&   0.1608E-08 & 4.00&   0.1610E-11 & 4.96 \\
\hline
\end{tabular}
\end{center}
\end{table}

\subsection{Example 2}
Consider problem (\ref{pde}) with $\Omega=(0,1)^2$ and $a=\p{1&0\\0&1}$.
The source term $f$ and the boundary value $g$ are chosen so that the exact solution is
\an{\label{s2}
    u(x,y)=\sin(\pi x)\sin(\pi y). }
We use triangular meshes as shown in Figure \ref{g-4t} for Example 2.
The error and the order of convergence are listed in Table \ref{t2},
   where we have optimal order of convergence for $k\ge 1$ in both $L^2$ norm and
    $H^1$-like triple-bar norm.

\begin{figure}[!h]
    \centering
   {\includegraphics[width=0.26\linewidth]{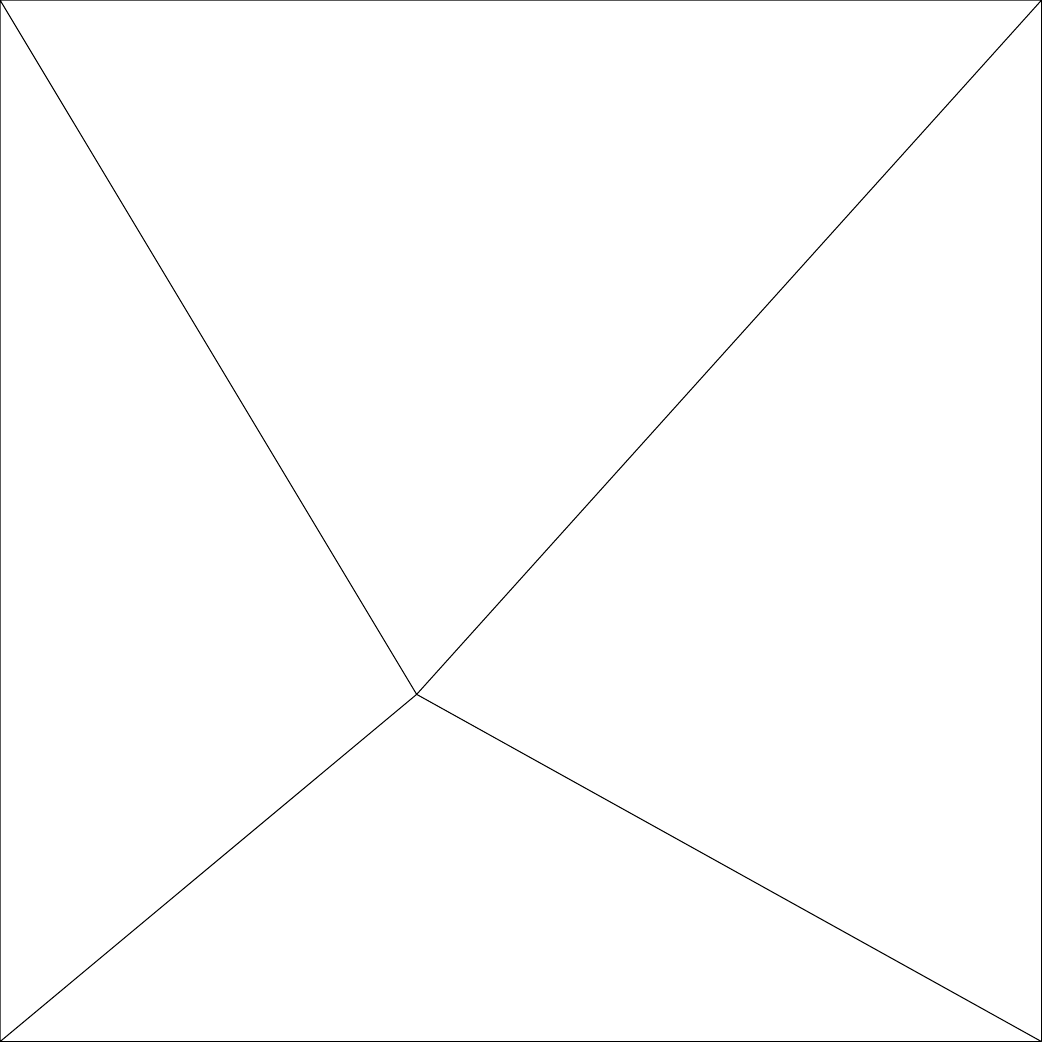}} \
   {\includegraphics[width=0.26\linewidth]{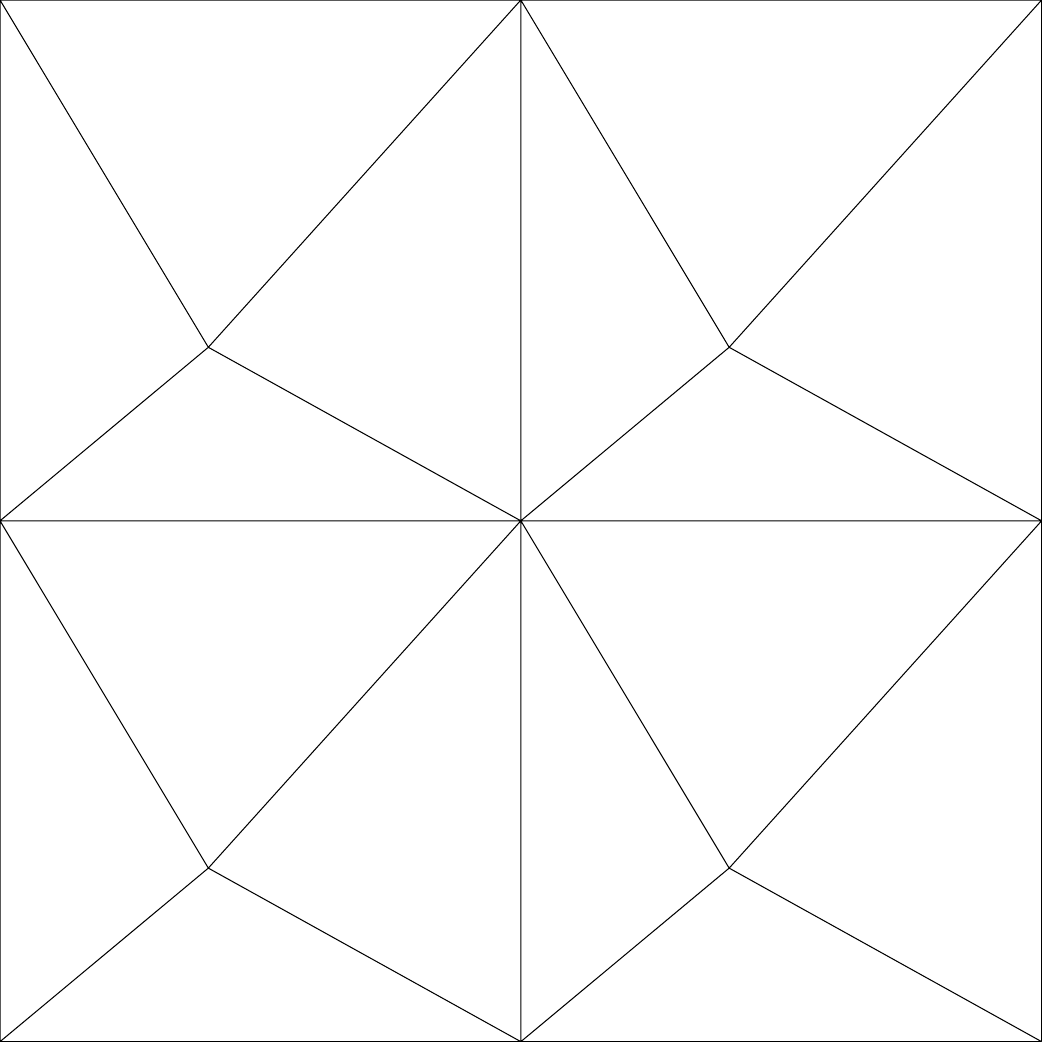}} \
   {\includegraphics[width=0.26\linewidth]{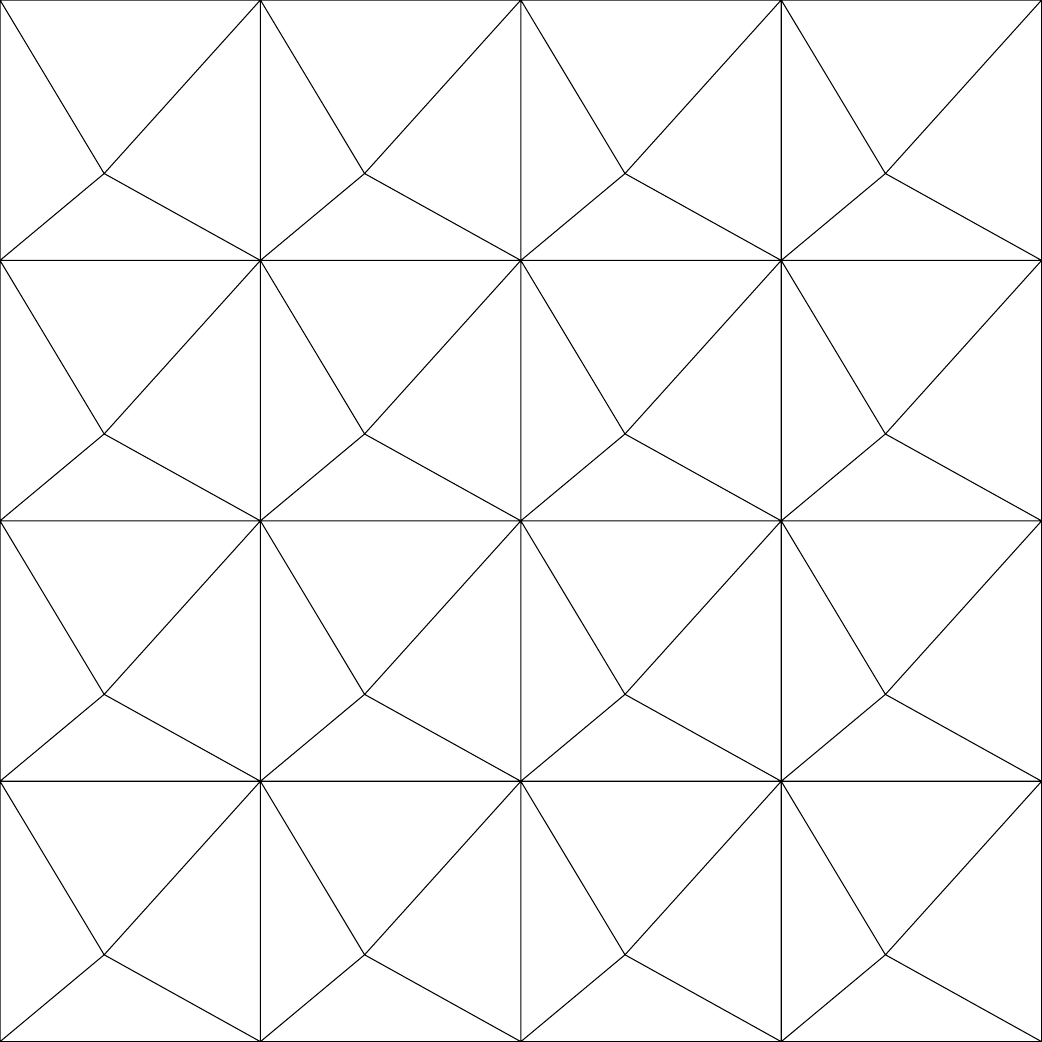}}
    \caption{Example 2. The first three triangular grids. }
    \label{g-4t}
\end{figure}

\begin{table}[h!]
    \caption{Example 2. The $P_k-P_{k-1}-[P_{k+1}]^2$ element,
      on rectangular grids shown in Figure \ref{g-4t}.}
    \label{t2}
    \begin{center}
    \begin{tabular}{c|c|cc|cc}  
        \hline
$k$&$\T_l$&$\3bar Q_hu-u_h\3bar$&Rate & $\|Q_hu-u_h\|$ & Rate \\
        \hline
  & 6&    0.120E-02 & 2.01&    0.141E-01 & 2.00 \\
1 & 7&    0.300E-03 & 2.00&    0.354E-02 & 2.00 \\
  & 8&    0.749E-04 & 2.00&    0.885E-03 & 2.00 \\
\hline
  & 6&   0.5734E-03 & 2.00&   0.1482E-05 & 3.00 \\
2 & 7&   0.1434E-03 & 2.00&   0.1852E-06 & 3.00 \\
  & 8&   0.3584E-04 & 2.00&   0.2314E-07 & 3.00 \\
\hline
  & 5&   0.3645E-04 & 3.00&   0.1360E-06 & 3.99 \\
3 & 6&   0.4559E-05 & 3.00&   0.8517E-08 & 4.00 \\
  & 7&   0.5699E-06 & 3.00&   0.5326E-09 & 4.00 \\
\hline
  & 5&   0.4222E-06 & 4.00&   0.9119E-09 & 5.01 \\
4 & 6&   0.2639E-07 & 4.00&   0.2846E-10 & 5.00 \\
  & 7&   0.1650E-08 & 4.00&   0.1110E-11 & 4.68 \\
\hline
\end{tabular}
\end{center}
\end{table}

\subsection{Example 3}
Consider problem (\ref{pde}) with $\Omega=(0,1)^2$ and $a=\p{1&0\\0&1}$.
The source term $f$ and the boundary value $g$ are chosen so that the exact solution is
\an{\label{s3}
    u(x,y)=e^{\pi x}\cos(\pi y). }
We use rectangular meshes as shown in Figure \ref{g-square} for Example 3.
The error and the order of convergence are listed in Table \ref{t3},
   where we have one order supconvergence for $k=1$ in  $H^1$-like triple-bar norm,
    one order supconvergence for $k=2$ in both $L^2$ norm and
    $H^1$-like triple-bar norm,
      and one order supconvergence for $k=3$ in both $L^2$ norm and
    $H^1$-like triple-bar norm.

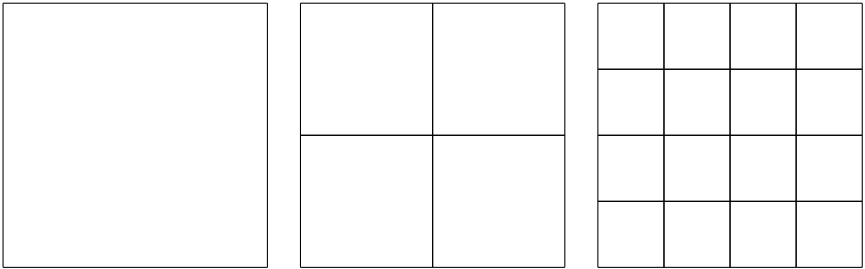
\begin{figure}[h!]
 \begin{center} \setlength\unitlength{1.25pt}
\begin{picture}(260,80)(0,0)
  \def\tr{\begin{picture}(20,20)(0,0)\put(0,0){\line(1,0){20}}\put(0,20){\line(1,0){20}}
          \put(0,0){\line(0,1){20}} \put(20,0){\line(0,1){20}}
   \end{picture}}
 {\setlength\unitlength{5pt}
 \multiput(0,0)(20,0){1}{\multiput(0,0)(0,20){1}{\tr}}}

  {\setlength\unitlength{2.5pt}
 \multiput(45,0)(20,0){2}{\multiput(0,0)(0,20){2}{\tr}}}

  \multiput(180,0)(20,0){4}{\multiput(0,0)(0,20){4}{\tr}}

 \end{picture}\end{center}
\caption{Example 3.  The first three rectangular grids.}
\label{g-square}
\end{figure}

\begin{table}[h!]
    \caption{Example 3. The $P_k-P_{k-1}-[P_{k+1}]^2$ element,
      on rectangular grids shown in Figure \ref{g-square}.}
    \label{t3}
    \begin{center}
    \begin{tabular}{c|c|cc|cc}  
        \hline
$k$&$\T_l$&$\3bar Q_hu-u_h\3bar$&Rate & $\|Q_hu-u_h\|$ & Rate \\
        \hline
  & 6&    0.141E-01 & 2.00&    0.120E-02 & 2.01 \\
1 & 7&    0.354E-02 & 2.00&    0.300E-03 & 2.00 \\
  & 8&    0.885E-03 & 2.00&    0.749E-04 & 2.00 \\
\hline
  & 6&    0.158E-03 & 3.00&    0.113E-05 & 4.00 \\
2 & 7&    0.197E-04 & 3.00&    0.709E-07 & 4.00 \\
  & 8&    0.246E-05 & 3.00&    0.444E-08 & 4.00 \\
\hline
  & 4&    0.251E-03 & 4.80&    0.409E-05 & 5.28 \\
3 & 5&    0.143E-04 & 4.13&    0.128E-06 & 4.99 \\
  & 6&    0.889E-06 & 4.01&    0.407E-08 & 4.98 \\
\hline
\end{tabular}
\end{center}
\end{table}

\subsection{Example 4}
Consider problem (\ref{pde}) with $\Omega=(0,1)^2$ and $a=\p{1&0\\0&1}$.
The source term $f$ and the boundary value $g$ are chosen so that the exact solution is
\an{\label{s4}
    u(x,y)=e^{2x-1}(y-y^3). }
We use polygonal meshes (mixing dodecagons(12 sided) and heptagons(7 sided))
   as shown in Figure \ref{g-12} for Example 4.
The error and the order of convergence are listed in Table \ref{t3},
   where we have optimal order convergence for all $k\ge 1$ in  in both $L^2$ norm and
    $H^1$-like triple-bar norm.

\begin{figure}[!h]
    \centering
   {\includegraphics[width=0.26\linewidth]{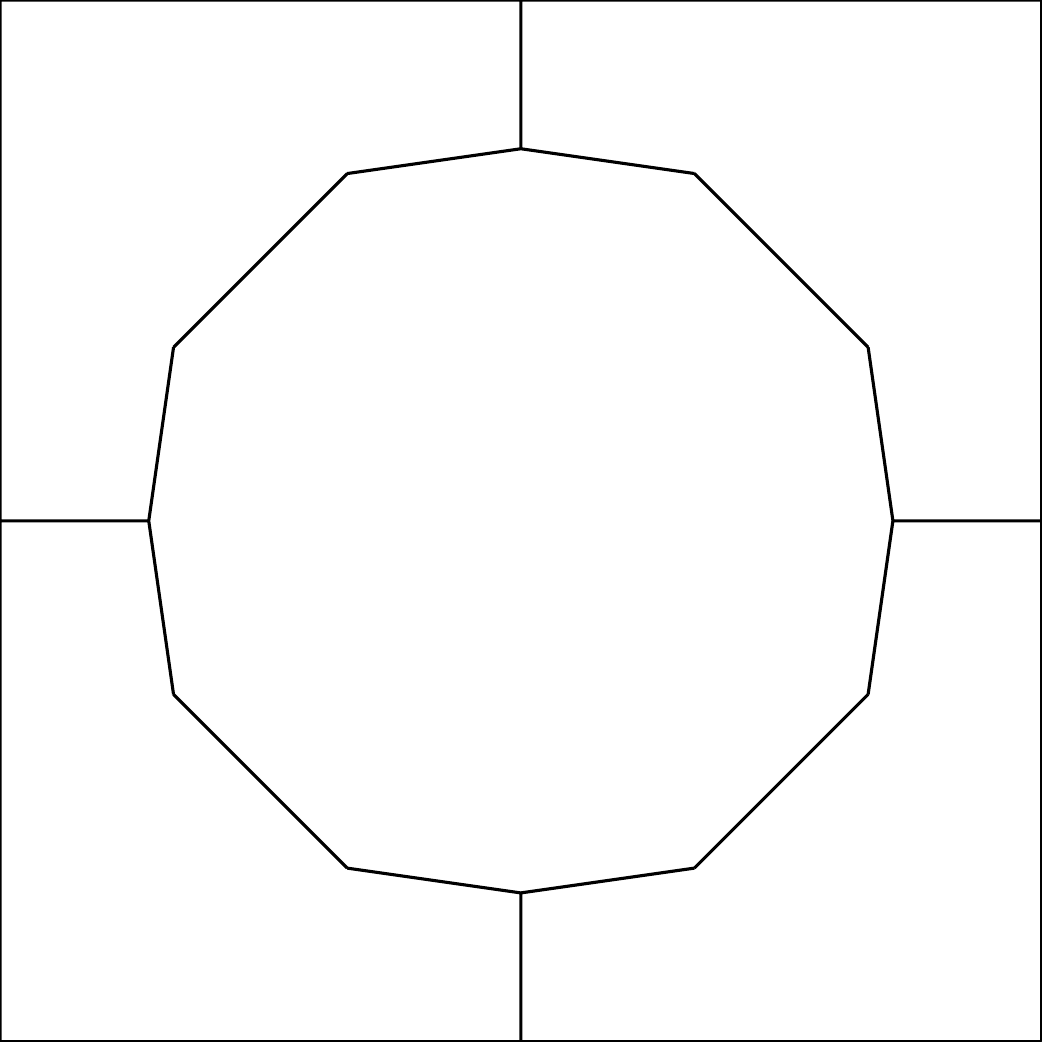}} \
   {\includegraphics[width=0.26\linewidth]{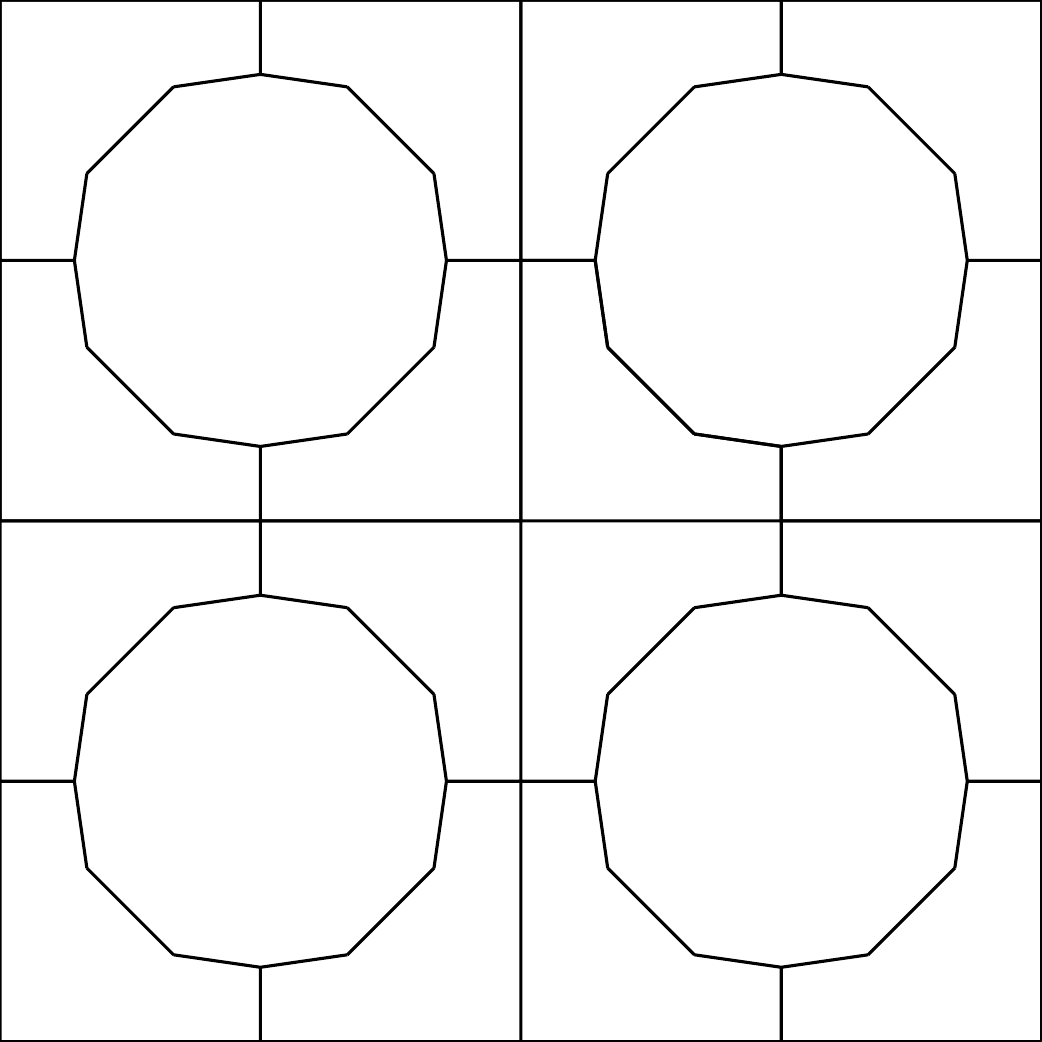}} \
   {\includegraphics[width=0.26\linewidth]{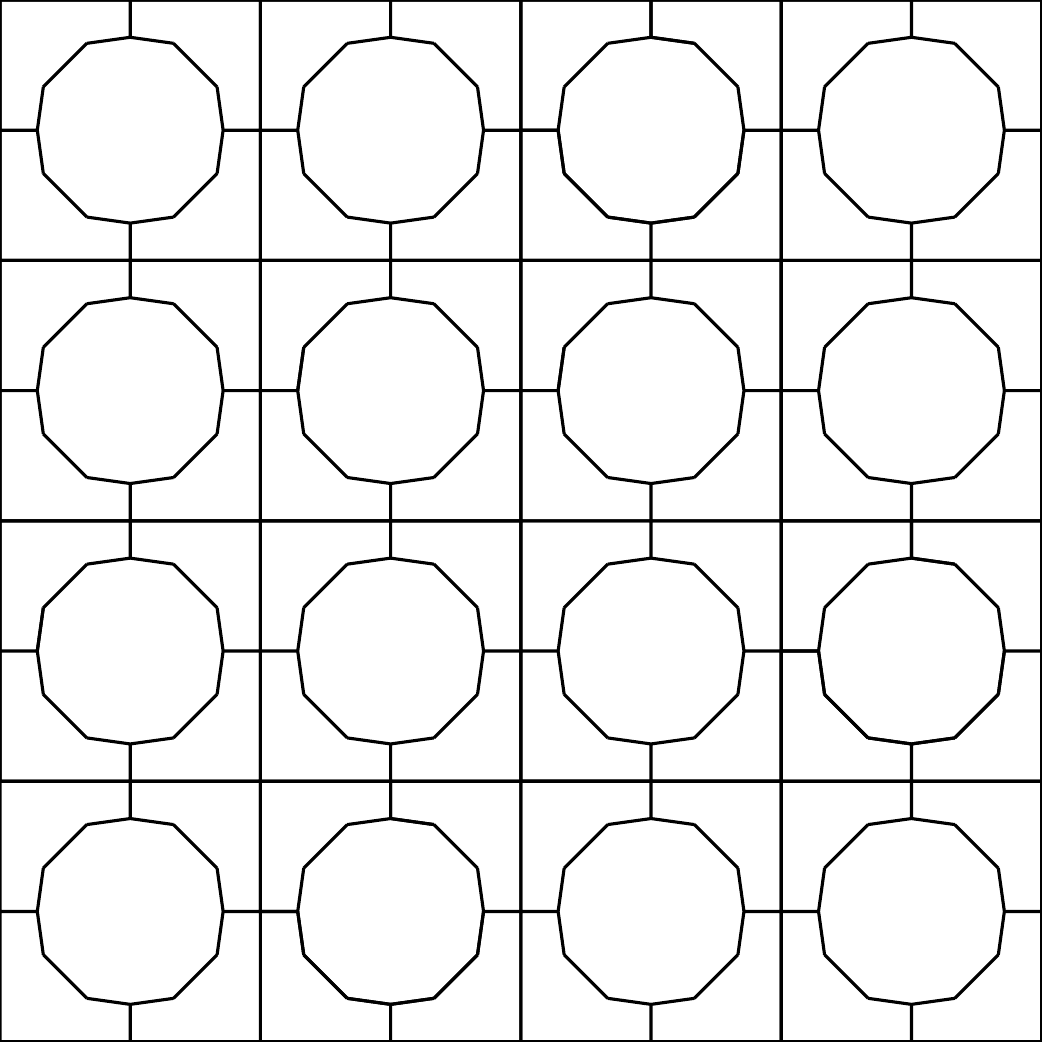}}
    \caption{Example 4. The first three polygonal grids. }
    \label{g-12}
\end{figure}

\begin{table}[h!]
    \caption{Example 4. The $P_k-P_{k-1}-[P_{k+2}]^2$ element,
      on polygonal grids shown in Figure \ref{g-12}.}
    \label{t4}
    \begin{center}
    \begin{tabular}{c|c|cc|cc}
        \hline
$k$&$\T_l$&$\3bar Q_hu-u_h\3bar$&Rate & $\|Q_hu-u_h\|$ & Rate \\
        \hline
  & 6&   0.3735E-01 & 1.00&   0.7856E-04 & 2.00\\
1 & 7&   0.1868E-01 & 1.00&   0.1966E-04 & 2.00\\
  & 8&   0.9339E-02 & 1.00&   0.4916E-05 & 2.00\\
\hline
  & 5&   0.1504E-02 & 1.98&   0.3242E-05 & 2.95\\
2 & 6&   0.3782E-03 & 1.99&   0.4131E-06 & 2.97\\
  & 7&   0.9482E-04 & 2.00&   0.5216E-07 & 2.99\\
\hline
  & 4&   0.1267E-03 & 2.97&   0.4106E-06 & 3.95\\
3 & 5&   0.1600E-04 & 2.99&   0.2636E-07 & 3.96\\
  & 6&   0.2010E-05 & 2.99&   0.1673E-08 & 3.98\\
\hline
  & 2&   0.5517E-03 & 3.98&   0.5905E-05 & 5.26\\
4 & 3&   0.3518E-04 & 3.97&   0.1699E-06 & 5.12\\
  & 4&   0.2234E-05 & 3.98&   0.5253E-08 & 5.02\\
\hline
\end{tabular}
\end{center}
\end{table}

\end{document}